\documentclass{amsart}
\usepackage{amsmath}
\usepackage{amsfonts}
\usepackage{indentfirst}

\newtheorem{theorem}{Theorem}
\theoremstyle{plain}

\newtheorem{corollary}{Corollary}

\newtheorem{definition}{Definition}

\newtheorem{lemma}{Lemma}

\newtheorem{proposition}{Proposition}

\numberwithin{equation}{section}

\begin{document}
\title[On Harmonically-Convex Functions]{Some Hermite-Hadamard-Fejer type
inequalities for Harmonically convex functions via Fractional Integral}
\author{\.{I}mdat \.{I}\c{s}can}
\address{Department of Mathematics, Faculty of Arts and Sciences,\\
Giresun University, 28100, Giresun, Turkey.}
\email{imdati@yahoo.com, imdat.iscan@giresun.edu.tr}
\author{Sercan Turhan}
\address{Dereli Vocational High School,\\
Giresun University, 28100, Giresun, Turkey.}
\email{sercanturhan28@gmail.com, sercan.turhan@giresun.edu.tr}
\author{Selahatt\.{I}n Maden}
\address{Department of Mathematics, Faculty of Arts and Sciences,\\
Ordu University, 52000, Ordu, Turkey.}
\email{maden55@mynet.com}

\begin{abstract}
In this paper, we gave the new general identity for differentiable
functions. As a result of this identity some new and general inequalities
for differentiable harmonically-convex functions are obtained.
\end{abstract}

\maketitle

\section{Introduction}

The classical or the usual convexity is defined as follows:

A function $f:I\longrightarrow 
%TCIMACRO{\U{211d} }%
%BeginExpansion
\mathbb{R}
%EndExpansion
,\ \emptyset \neq I\subseteq 
%TCIMACRO{\U{211d} }%
%BeginExpansion
\mathbb{R}
%EndExpansion
$, is said to be convex on $I$ if inequality 
\begin{equation*}
f\left( tx+\left( 1-t\right) y\right) \leq tf(x)+\left( 1-t\right) f(y)
\end{equation*}%
holds for all $x,y\in I$ and $t\in \left[ 0,1\right] $.

A number of papers have been written on inequalities using the classical
convexty and one of the most captivating inequalities in mathematical
analysis is stated as follows 
\begin{equation}
f\left(\frac{a+b}{2}\right) \leq \frac{1}{b-a}\int\limits_{a}^{b}f(x)dx\leq 
\frac{f(a)+f(b)}{2}\text{,}  \label{1.1}
\end{equation}
where $f:I\subseteq 
%TCIMACRO{\U{211d} }%
%BeginExpansion
\mathbb{R}
%EndExpansion
\longrightarrow $ be a convex mapping and $a, b \in I $ with $a \le b$ .
Both the inequalities hold in reversed direction if $f$ is concave. The
inequalities stated in (\ref{1.1}) are known as Hermite-Hadamard
inequalities.

For more results on (\ref{1.1}) which provide new proof, significantly
extensions, generalizations, refinements, counterparts, new
Hermite-Hadamard-type inequalities and numerous applications, we refer the
interested reader to \cite{ZD10, SSD12, IK215, I113, I114, I213,IS15, ML14,
MS12, SSYB13} and the references therein.

The usual notion of convex function have been generalized in diverse
manners. One of them is the so called harmonically s-convex functions and is
stated in the defination below.

\begin{definition}
\label{dfn1}\cite{IK215, I115}Let $I\subset \left( 0,\infty \right) $ be a
real interval. A function $f:I\longrightarrow 
%TCIMACRO{\U{211d} }%
%BeginExpansion
\mathbb{R}
%EndExpansion
$ is said to be harmonically s-convex(concave), if%
\begin{equation*}
f\left( \frac{xy}{tx+(1-t)y}\right) \leq \left( \geq \right) t^{s}f\left(
y\right) +\left( 1-t\right) ^{s}f\left( x\right)
\end{equation*}%
holds for all $x,y\in I$ and $t\in \left[ 0,1\right] $, and for some fixed $%
s\in \left( 0,1\right] .$
\end{definition}

It can be easily seen that for $s=1$ in Defination 1 reduces to following
Defination 2:

\begin{definition}
\label{dfn2}\cite{I113} A function $f:I\subseteq 
%TCIMACRO{\U{211d} }%
%BeginExpansion
\mathbb{R}
%EndExpansion
\backslash \{0\}\longrightarrow 
%TCIMACRO{\U{211d} }%
%BeginExpansion
\mathbb{R}
%EndExpansion
$ is said to be harmonically-convex function$,$ if 
\begin{equation*}
f\left( \frac{xy}{tx+(1-t)y}\right) \leq tf\left( y\right) +\left(
1-t\right) f\left( x\right)
\end{equation*}%
holds for all $x,y\in I$ and $t\in \left[ 0,1\right] $ . If the inequality
is reversed, then $f$ is said to be harmonically concave.
\end{definition}

\begin{proposition}
\cite{I113} Let I $\subset $ $%
%TCIMACRO{\U{211d} }%
%BeginExpansion
\mathbb{R}
%EndExpansion
\backslash \{0\}$ be a real interval and $f:I\rightarrow 
%TCIMACRO{\U{211d} }%
%BeginExpansion
\mathbb{R}
%EndExpansion
$ is function, then:
\end{proposition}

if $I\subset (0,\infty )$ and $f$ is convex and nondecreasing function then
f is harmonically convex.

if $I\subset (0,\infty )$ and $f$ is harmonically convex and nonincreasing
function then f is convex.

if $I\subset (-\infty ,0)$ and $f$ is harmonically convex and nondecreasing
function then f is convex.

if $I\subset (-\infty ,0)$ and $f$ is convex and nonincreasing function then
f is harmonically convex.

For the properties of harmonically-convex functions and
harmonically-s-convex function, we refer the reader to \cite{CW141, IK215,
I113, I115, I114, IK15, IW14} and the reference there in.

Most recently, a number of findings have been seen on Hermite-Hadamard type
integral inequalities for harmonically-convex and for harmonically-s-convex
functions.

In \cite{LF06}, Fej\'{e}r established the following Fej\'{e}r inequality
which is the weighted generalization of Hermite-Hadamard inequality (\ref%
{1.1}):

\begin{theorem}
Let $f:[a,b]\longrightarrow 
%TCIMACRO{\U{211d} }%
%BeginExpansion
\mathbb{R}
%EndExpansion
$ be convex function. Then the inequality
\end{theorem}

\begin{equation}
f\left( \frac{a+b}{2}\right) \int\limits_{a}^{b}g(x)dx\leq
\int\limits_{a}^{b}f(x)g(x)dx\leq \frac{f(a)+f(b)}{2}\int%
\limits_{a}^{b}g(x)dx\text{,}  \label{1.2}
\end{equation}

holds, where $g:[a,b]\longrightarrow 
%TCIMACRO{\U{211d} }%
%BeginExpansion
\mathbb{R}
%EndExpansion
$ is nonnegative, integrable and symmetric to $(a+b)/2$.

For some results which generalize, improve, and extend the inequalities (\ref%
{1.1}) and (\ref{1.2}) see \cite{MS12}.

In \cite{I113}, \.{I}\c{s}can gave defination of harmonically convex
functions and established following Hermite- Hadamard type inequality for
harmonically convex functions as follows:

\begin{theorem}
\cite{MS12} Let $f:I\subset \mathbb{R}\backslash \{0\}\mathbb{\rightarrow R}$
be a harmonically convex function and $a,b\in I$ with $a<b$ . If $\ f\in L%
\left[ a,b\right] $ then the following inequalities hold:%
\begin{equation}
f\left( \frac{2ab}{a+b}\right) \leq \frac{ab}{b-a}\int\limits_{a}^{b}\frac{%
f(x)}{x^{2}}dx\leq \frac{f(a)+f(b)}{2}.  \label{1.3}
\end{equation}

In \cite{IW14}, \.{I}\c{s}can and Wu represented Hermite-Hadamard's
inequalities for harmonically convex functions in fractional integral form
as follows:
\end{theorem}

\begin{theorem}
\cite{IW14} Let $f:I\subseteq \mathbb{R}^{+}\mathbb{\rightarrow R}$ be a
function such that $f\in L\left[ a,b\right] $, where $a,b\in I$ with $a<b.$
If $f$ is harmonically-convex on $\left[ a,b\right] $ , then the following
inequalities for fractional integrals hold: 
\begin{equation}
f\left( \frac{2ab}{a+b}\right) \leq \frac{\Gamma (\alpha +1)}{2}\left( \frac{%
ab}{b-a}\right) ^{\alpha }\left\{ 
\begin{array}{c}
J_{1/a^{-}}^{\alpha }(f\circ h)(1/b) \\ 
+J_{1/b^{+}}^{\alpha }(f\circ h)(1/a)%
\end{array}%
\right\}  \label{1.4}
\end{equation}%
\begin{equation*}
\leq \frac{f(a)+f(b)}{2}.
\end{equation*}

with $\alpha >0$ and $h(x)=1/x$.
\end{theorem}

\begin{definition}
A function $g:\left[ a,b\right] \subseteq 
%TCIMACRO{\U{211d} }%
%BeginExpansion
\mathbb{R}
%EndExpansion
\backslash \{0\}\longrightarrow 
%TCIMACRO{\U{211d} }%
%BeginExpansion
\mathbb{R}
%EndExpansion
$ is said to be harmonically symmetric with respect to $2ab/a+b$ if%
\begin{equation}
g(x)=g\left( \frac{1}{\frac{1}{a}+\frac{1}{b}-\frac{1}{x}}\right)
\label{1.5}
\end{equation}
\end{definition}

holds for all $x\in \lbrack a,b].$

In \cite{CW141} Chan and Wu represented Hermite-Hadamard-Fejer inequality
for harmonically convex functions as follows:

\begin{theorem}
Suppose that $f:I\subseteq 
%TCIMACRO{\U{211d} }%
%BeginExpansion
\mathbb{R}
%EndExpansion
\backslash \{0\}\longrightarrow 
%TCIMACRO{\U{211d} }%
%BeginExpansion
\mathbb{R}
%EndExpansion
$ be harmonically-convex function and $a,b\in I$, with $a<b$. If $f\in L%
\left[ a,b\right] $ and $g:[a,b]\subseteq 
%TCIMACRO{\U{211d} }%
%BeginExpansion
\mathbb{R}
%EndExpansion
\backslash \{0\}\longrightarrow 
%TCIMACRO{\U{211d} }%
%BeginExpansion
\mathbb{R}
%EndExpansion
$ is nonnegative, integrable and harmonically symmetric with respect to $%
2ab/a+b,$then 
\begin{equation}
f\left( \frac{2ab}{a+b}\right) \int\limits_{a}^{b}\frac{g(x)}{x^{2}}dx\leq
\int\limits_{a}^{b}\frac{f(x)g(x)}{x^{2}}dx\leq \frac{f(a)+f(b)}{2}%
\int\limits_{a}^{b}\frac{g(x)}{x^{2}}dx  \label{1.6}
\end{equation}%
In \cite{IK15} \.{I}\c{s}can and Kunt represented Hermite-Hadamard-Fejer
type inequality for harmonically convex functions in fractional integral
forms and established following identity as follows:
\end{theorem}

\begin{theorem}
Let $f:[a,b]\longrightarrow 
%TCIMACRO{\U{211d} }%
%BeginExpansion
\mathbb{R}
%EndExpansion
$ be harmonically convex function with $a<b$ and $f\in L\left[ a,b\right] $.
If $g:[a,b]\longrightarrow 
%TCIMACRO{\U{211d} }%
%BeginExpansion
\mathbb{R}
%EndExpansion
$ is nonnegative, integrable and harmonically symmetric with respect to $%
2ab/a+b,$then the following inequalities for fractional integrals hold: 
\begin{equation}
f\left( \frac{2ab}{a+b}\right) \left[ J_{1/a^{-}}^{\alpha }(g\circ
h)(1/b)+J_{1/b^{+}}^{\alpha }(g\circ h)(1/a)\right]   \label{1.7}
\end{equation}%
\begin{equation*}
\leq \left[ J_{1/a^{-}}^{\alpha }(fg\circ h)(1/b)+J_{1/b^{+}}^{\alpha
}(fg\circ h)(1/a)\right] 
\end{equation*}%
\begin{equation*}
\leq \frac{f(a)+f(b)}{2}\left[ J_{1/a^{-}}^{\alpha }(g\circ
h)(1/b)+J_{1/b^{+}}^{\alpha }(g\circ h)(1/a)\right] 
\end{equation*}
\end{theorem}

with $\alpha >0$ and $h(x)=1/x,$ $x\in \left[ \frac{1}{b},\frac{1}{a}\right] 
$.

\begin{definition}
Let $f\in L[a,b]$. The right-hand side and left-hand side Hadamard
fractional integrals $J_{a^{+}}^{\alpha }f$ and $J_{b^{-}}^{\alpha }f$ of
order $\alpha >0$ with $a\geq 0$ are defined by

\begin{equation*}
J_{a^{+}}^{\alpha }f(x)=\frac{1}{\Gamma (\alpha )}\overset{x}{\underset{a}{%
\int }}\left( x-t\right) ^{\alpha -1}f(t)dt,\ x>a
\end{equation*}%
\begin{equation*}
J_{b^{-}}^{\alpha }f(x)=\frac{1}{\Gamma (\alpha )}\overset{b}{\underset{x}{%
\int }}\left( t-x\right) ^{\alpha -1}f(t)dt,\ x<b
\end{equation*}%
respectively where $\Gamma (\alpha )$ is the Gamma function defined by $%
\Gamma (\alpha )=\overset{\infty }{\underset{0}{\int }}e^{-t}t^{\alpha -1}$
and $J_{a^{+}}^{0}f(x)=J_{b^{-}}^{0}f(x)=f(x)$
\end{definition}

\begin{lemma}
For $0< \theta \leq 1 $ and $0< a \leq b $ we have 
\begin{equation*}
\left\vert a^{\theta}-b^{\theta} \right\vert \leq \left( b-a\right)^{\theta}.
\end{equation*}
\end{lemma}

In \cite{H11} D. Y. Hwang found out a new identity and by using this
identity, established a new inequalities. Then in \cite{IS15} \.{I}. \.{I}%
\c{s}can and S. Turhan used this identity for GA-convex functions and obtain
generalized new inequalities. In this paper, we established a new inequality
similar to inequality in \cite{IS15} and then we obtained some new and
general integral inequalities for differentiable harmonically-convex
functions using this lemma. The following sections, let the notion, $L\left(
t\right) =\frac{aH}{tH+(1-t)a}$, $U\left( t\right) =\frac{bH}{tH+(1-t)b}$
and $H=H\left( a,b\right) =\frac{2ab}{a+b}$.

\section{Main result}

Throughout this section, let $\left\Vert g\right\Vert _{\infty }=\sup_{t\in %
\left[ a,b\right] }\left\vert g(x)\right\vert $, for the continuous function 
$g:[a,b]\longrightarrow \lbrack 0,\infty )$ be differentiable mapping $I^{o}$%
, where $a,b\in I$ with $a\leq b$, and $h:\left[ a,b\right] \longrightarrow %
\left[ 0,\infty \right) $ be differentiable mapping.

\begin{lemma}
If $f^{\prime }\in L\left[ a,b\right] $then the following inequality holds:
\end{lemma}

\begin{equation}
\left[ h(b)-2h(a)\right] \frac{f(a)}{2}+h(b)\frac{f(b)}{2}%
-\int\limits_{a}^{b}f(x)h^{\prime }(x)dx\   \label{2.1}
\end{equation}%
\begin{equation*}
=\frac{b-a}{4ab}\left\{ \int\limits_{0}^{1}\left[ 2h\left( L(t)\right) -h(b)%
\right] f^{\prime }\left( L(t)\right) \left( L(t)\right) ^{2}dt\ \right.
\end{equation*}%
\begin{equation*}
\ \ \ \ \ \ \ \ \ \ \ \ \ \ \ \ \ \ \left. +\int\limits_{0}^{1}\left[
2h\left( U(t)\right) -h(b)\right] f^{\prime }\left( U(t)\right) \left(
U(t)\right) ^{2}dt\right\}
\end{equation*}

\begin{proof}
By the integration by parts, we have 
\begin{equation*}
I_{1}=\int\limits_{0}^{1}\left[ 2h\left( L(t)\right) -h(b)\right] d\left(
f\left( L(t)\right) \right)
\end{equation*}%
\begin{equation*}
=\left. \left[ 2h\left( L(t)\right) -h(b)\right] f\left( L(t)\right)
\right\vert _{0}^{1}
\end{equation*}%
\begin{equation*}
\ \ \ \ \ \ \ \ \ \ \ \ \ -\left( \frac{1}{a}-\frac{1}{b}\right)
\int\limits_{0}^{1}f\left( L(t)\right) h^{\prime }\left( L(t)\right) \left(
L(t)\right) ^{2}dt
\end{equation*}%
\newline
and

\begin{equation*}
I_{2}=\int\limits_{0}^{1}\left[ 2h\left( U(t)\right) -h(b)\right] d\left(
f\left( U(t)\right) \right)
\end{equation*}%
\begin{equation*}
=\left. \left[ 2h\left( U(t)\right) -h(b)\right] f\left( U(t)\right)
\right\vert _{0}^{1}
\end{equation*}%
\begin{equation*}
\ \ \ \ \ \ \ \ \ \ \ \ \ -\left( \frac{1}{a}-\frac{1}{b}\right)
\int\limits_{0}^{1}f\left( U(t)\right) h^{\prime }\left( U(t)\right) \left(
U(t)\right) ^{2}dt
\end{equation*}

Therefore

\begin{equation}
\frac{I_{1}+I_{2}}{2}=\left[ h(b)-2h(a)\right] \frac{f(a)}{2}+h(b)\frac{f(b)%
}{2}-\frac{b-a}{2ab}\left\{ \int\limits_{0}^{1}f\left( L(t)\right) h^{\prime
}\left( L(t)\right) \left( L(t)\right) ^{2}dt\right.
\end{equation}%
\begin{equation*}
\ \ \ \ \ \ \ \ \ \ \ \ \ \ \ \ \ \ \ \ \ \ \ \ \ \ \ \ \ \ \ \ \ \ \ \ \ \
\ \ \ \ \ \ \ \ \ \ \ \ \ \ \ \ \ \ \ \ \ \ \ \ \ \ \ \ \ \ \ \ \ \ \ \ \ \
\ \ \ \ \ \ \ \ \ \ \ \ \ \ \ \ \ \ \ \ \ \ \ \ \ \left.
+\int\limits_{0}^{1}f\left( U(t)\right) h^{\prime }\left( U(t)\right) \left(
U(t)\right) ^{2}dt\right\}
\end{equation*}

This complete the proof
\end{proof}

\begin{lemma}
For $a,H,b>0$, we have

\begin{equation}
\zeta _{1}\left( a,b\right) =\int\limits_{0}^{1}\left\vert 2h\left(
L(t)\right) -h(b)\right\vert \left( 1-t\right) \left( L(t)\right) ^{2}dt
\end{equation}%
\begin{equation}
\zeta _{2}\left( a,b\right) =\int\limits_{0}^{1}t\left( L(t)\right)
^{2}\left\vert 2h\left( L(t)\right) -h(b)\right\vert
dt+\int\limits_{0}^{1}t(\left( U(t)\right) ^{2}\left\vert 2h\left(
U(t)\right) -h(b)\right\vert dt
\end{equation}%
\begin{equation}
\zeta _{3}\left( a,b\right) =\int\limits_{0}^{1}\left\vert 2h\left(
U(t)\right) -h(b)\right\vert \left( 1-t\right) \left( U(t)\right) ^{2}dt
\end{equation}
\end{lemma}

\begin{theorem}
Let $f:I\subseteq 
%TCIMACRO{\U{211d} }%
%BeginExpansion
\mathbb{R}
%EndExpansion
=\left( 0,\infty \right) \longrightarrow 
%TCIMACRO{\U{211d} }%
%BeginExpansion
\mathbb{R}
%EndExpansion
$ be differentiable mapping $I^{o}$, where $a,b\in I$ with $a<b$. If the
mapping $\left\vert f^{\prime }\right\vert $ is harmonically-convex on $%
\left[ a,b\right] $, then the following inequality holds:

\begin{equation}
\left\vert \left[ h(b)-2h(a)\right] \frac{f(a)}{2}+h(b)\frac{f(b)}{2}%
-\int\limits_{a}^{b}f(x)h^{\prime }(x)dx\right\vert  \label{2.6}
\end{equation}%
\begin{equation*}
\leq \frac{b-a}{4ab}\left[ \zeta _{1}(a,b)\left\vert f^{\prime
}(a)\right\vert +\zeta _{2}(a,b)\left\vert f^{\prime }(H)\right\vert +\zeta
_{3}(a,b)\left\vert f^{\prime }(b)\right\vert \right]
\end{equation*}%
where $\zeta _{1}\left( a,b\right) ,\zeta _{2}\left( a,b\right) ,\zeta
_{3}\left( a,b\right) $ are defined in Lemma 3.
\end{theorem}

\begin{proof}
Continuing equality (\ref{2.1}) in Lemma 2 
\begin{equation}
\left\vert \left[ h(b)-2h(a)\right] \frac{f(a)}{2}+h(b)\frac{f(b)}{2}%
-\int\limits_{a}^{b}f(x)h^{\prime }(x)dx\ \right\vert  \label{2.7}
\end{equation}%
\begin{equation*}
\leq \frac{b-a}{4ab}\left\{ \int\limits_{0}^{1}\left\vert 2h\left(
L(t)\right) -h(b)\right\vert \left\vert f^{\prime }\left( L(t)\right) \left(
L(t)\right) ^{2}\right\vert dt\ \right.
\end{equation*}%
\begin{equation*}
\ \ \ \ \ \ \ \ \ \ \ \ \ \ \ \ \ \ \left. +\int\limits_{0}^{1}\left\vert
2h\left( U(t)\right) -h(b)\right\vert \left\vert f^{\prime }\left(
U(t)\right) \left( U(t)\right) ^{2}\right\vert dt\right\}
\end{equation*}

Using $\left\vert f^{\prime }\right\vert $ is harmoncally-convex in (\ref%
{2.7})

\begin{equation}
\left\vert \left[ h(b)-2h(a)\right] \frac{f(a)}{2}+h(b)\frac{f(b)}{2}%
-\int\limits_{a}^{b}f(x)h^{\prime }(x)dx\ \right\vert  \label{2.8}
\end{equation}%
\begin{equation*}
\leq \frac{b-a}{4ab}\left\{ \int\limits_{0}^{1}\left\vert 2h\left(
L(t)\right) -h(b)\right\vert \left\{ t\left\vert f^{\prime }(H)\right\vert
+(1-t)\left\vert f^{\prime }(a)\right\vert \right\} \left( L(t)\right)
^{2}dt\ \right.
\end{equation*}%
\begin{equation*}
\ \ \ \ \ \ \ \ \ \ \ \ \ \ \ \ \ \ \left. +\int\limits_{0}^{1}\left\vert
2h\left( U(t)\right) -h(b)\right\vert \left\{ t\left\vert f^{\prime
}(H)\right\vert +(1-t)\left\vert f^{\prime }(b)\right\vert \right\} \left(
U(t)\right) ^{2}dt\right\} ,
\end{equation*}

by (\ref{2.8}) and Lemma 2, this proof is complete.
\end{proof}

\begin{corollary}
Let $h(t)=\int\limits_{1/t}^{1/a}\left[ \left( x-\frac{1}{b}\right) ^{\alpha
-1}+\left( \frac{1}{a}-x\right) ^{\alpha -1}\right] g\circ \varphi (x)dx$
for all $1/t\in \lbrack \frac{1}{b},\frac{1}{a}]$ , $\alpha >0$ and $g:\left[
a,b\right] \longrightarrow \left[ 0,\infty \right) $ be continuous positive
mapping and symmetric to $\frac{2ab}{a+b}$ in Teorem 7, we obtain:

\begin{equation}
\left\vert \left( \frac{f(a)+f(b)}{2}\right) \left[ J_{1/b^{+}}^{\alpha
}g\circ \varphi (1/a)+J_{1/a^{-}}^{\alpha }g\circ \varphi (1/b)\right] -%
\left[ J_{1/b^{+}}^{\alpha }\left( fg\circ \varphi \right) \left( 1/a\right)
+J_{1/a^{-}}^{\alpha }\left( fg\circ \varphi \right) \left( 1/b\right) %
\right] \right\vert  \label{2.9}
\end{equation}%
\begin{equation*}
\leq \frac{(b-a)^{\alpha +1}\left\Vert g\right\Vert _{\infty }}{2^{\alpha
+1}(ab)^{\alpha +1}\Gamma \left( \alpha +1\right) }\left[ C_{1}\left( \alpha
\right) \left\vert f^{\prime }(a)\right\vert +C_{2}\left( \alpha \right)
\left\vert f^{\prime }(H)\right\vert +C_{3}\left( \alpha \right) \left\vert
f^{\prime }(b)\right\vert \right]
\end{equation*}

where 
\begin{equation*}
C_{1}\left( \alpha \right) =\int\limits_{0}^{1}(1-t)\left[ (1+t)^{\alpha
}-(1-t)^{\alpha }\right] (L(t))^{2}dt
\end{equation*}%
\begin{equation*}
C_{2}\left( \alpha \right) =\int\limits_{0}^{1}t\left[ (1+t)^{\alpha
}-(1-t)^{\alpha }\right] \left[ (L(t))^{2}+(U(t))^{2}\right] dt
\end{equation*}%
\begin{equation*}
C_{3}\left( \alpha \right) =\int\limits_{0}^{1}(1-t)\left[ (1+t)^{\alpha
}-(1-t)^{\alpha }\right] (L(t))^{2}dt
\end{equation*}

Specially in (\ref{2.9}) and using Lemma 1, for $0<\alpha \leq 1$ we have:

\begin{equation}
\left\vert \left( \frac{f(a)+f(b)}{2}\right) \left[ J_{1/b^{+}}^{\alpha
}g\circ \varphi (1/a)+J_{1/a^{-}}^{\alpha }g\circ \varphi (1/b)\right] -%
\left[ J_{1/b^{+}}^{\alpha }\left( fg\circ \varphi \right) \left( 1/a\right)
+J_{1/a^{-}}^{\alpha }\left( fg\circ \varphi \right) \left( 1/b\right) %
\right] \right\vert  \label{2.10}
\end{equation}%
\begin{equation*}
\leq \frac{(b-a)^{\alpha +1}\left\Vert g\right\Vert _{\infty }}{%
2(ab)^{\alpha +1}\Gamma \left( \alpha +1\right) }\left[ C_{1}\left( \alpha
\right) \left\vert f^{\prime }(a)\right\vert +C_{2}\left( \alpha \right)
\left\vert f^{\prime }(H)\right\vert +C_{3}\left( \alpha \right) \left\vert
f^{\prime }(b)\right\vert \right]
\end{equation*}

where 
\begin{equation*}
C_{1}\left( \alpha \right) =\int\limits_{0}^{1}(1-t)t^{\alpha }(L(t))^{2}dt
\end{equation*}%
\begin{equation*}
C_{2}\left( \alpha \right) =\int\limits_{0}^{1}t^{\alpha +1}\left[
(L(t))^{2}+(U(t))^{2}\right] dt
\end{equation*}%
\begin{equation*}
C_{3}\left( \alpha \right) =\int\limits_{0}^{1}(1-t)t^{\alpha }(U(t))^{2}dt
\end{equation*}
\end{corollary}

\begin{proof}
By left side of inequality (\ref{2.8}) in Teorem 7, when we write $%
h(t)=\int\limits_{1/t}^{1/a}\left[ \left( x-\frac{1}{b}\right) ^{\alpha
-1}+\left( \frac{1}{a}-x\right) ^{\alpha -1}\right] g\circ \varphi (x)dx$
for all $x\in \lbrack 1/b,1/a]$ and $\varphi (x)=1/x$, we have

\begin{equation*}
\left\vert 
\begin{array}{c}
\Gamma (\alpha )\left( \frac{f(a)+f(b)}{2}\right) \left[ J_{1/b^{+}}^{\alpha
}g\circ \varphi (1/a)+J_{1/a^{-}}^{\alpha }g\circ \varphi (1/b)\right] \\ 
-\Gamma (\alpha )\left[ J_{1/b^{+}}^{\alpha }\left( fg\circ \varphi \right)
\left( 1/a\right) +J_{1/a^{-}}^{\alpha }\left( fg\circ \varphi \right)
\left( 1/b\right) \right]%
\end{array}%
\right\vert
\end{equation*}%
.

On the other hand, right side of inequality (\ref{2.8})

\begin{equation}
\leq \frac{b-a}{4ab}\left\{ \int\limits_{0}^{1}\left\vert 
\begin{array}{c}
2\int\limits_{1/L(t)}^{1/a}\left[ \left( x-\frac{1}{b}\right) ^{\alpha
-1}+\left( \frac{1}{a}-x\right) ^{\alpha -1}\right] g\circ \varphi (x)dx \\ 
-\int\limits_{1/b}^{1/a}\left[ \left( x-\frac{1}{b}\right) ^{\alpha
-1}+\left( \frac{1}{a}-x\right) ^{\alpha -1}\right] g\circ \varphi (x)dx%
\end{array}%
\right\vert \left\{ t\left\vert f^{\prime }(H)\right\vert +(1-t)\left\vert
f^{\prime }(a)\right\vert \right\} (L(t))^{2}dt\ \right.  \label{2.11}
\end{equation}%
\begin{equation*}
+\left. \int\limits_{0}^{1}\left\vert 
\begin{array}{c}
2\int\limits_{1/U(t)}^{1/a}\left[ \left( x-\frac{1}{b}\right) ^{\alpha
-1}+\left( \frac{1}{a}-x\right) ^{\alpha -1}\right] g\circ \varphi (x)dx \\ 
-\int\limits_{1/b}^{1/a}\left[ \left( x-\frac{1}{b}\right) ^{\alpha
-1}+\left( \frac{1}{a}-x\right) ^{\alpha -1}\right] g\circ \varphi (x)dx%
\end{array}%
\right\vert \left\{ t\left\vert f^{\prime }(H)\right\vert +(1-t)\left\vert
f^{\prime }(b)\right\vert \right\} (U(t))^{2}dt\right\}
\end{equation*}%
.

Since $g(x)$ is symmetric to $x=\frac{2ab}{a+b}$, we have 
\begin{equation}
\left\vert 2\int\limits_{1/L(t)}^{1/a}\left[ \left( x-\frac{1}{b}\right)
^{\alpha -1}+\left( \frac{1}{a}-x\right) ^{\alpha -1}\right] g\circ \varphi
(x)dx-\int\limits_{1/b}^{1/a}\left[ \left( x-\frac{1}{b}\right) ^{\alpha
-1}+\left( \frac{1}{a}-x\right) ^{\alpha -1}\right] \left( g\circ \varphi
\right) (x)dx\right\vert  \label{2.12}
\end{equation}%
\begin{equation*}
\ \ \ \ \ \ \ \ \ \ \ \ \ \ \ \ \ \ \ \ \ \ \ \ \ \ \ \ \ \ \ \ \ \ \ \ \ \
\ \ \ \ \ \ \ \ \ \ \ \ \ \ \ \ \ \ \ \ \ \ \ \ \ \ \ \ \ \ \ \ \ \ \ \ \ \
\ \ \ \ \ \ \ \ \ \ \ \ \ \ \ =\left\vert \int\limits_{1/U(t)}^{1/L(t)}\left[
\left( x-\frac{1}{b}\right) ^{\alpha -1}+\left( \frac{1}{a}-x\right)
^{\alpha -1}\right] \left( g\circ \varphi \right) (x)dx\right\vert
\end{equation*}%
and

\begin{equation}
\left\vert 2\int\limits_{1/U(t)}^{1/a}\left[ \left( x-\frac{1}{b}\right)
^{\alpha -1}+\left( \frac{1}{a}-x\right) ^{\alpha -1}\right] g\circ \varphi
(x)dx-\int\limits_{1/b}^{1/a}\left[ \left( x-\frac{1}{b}\right) ^{\alpha
-1}+\left( \frac{1}{a}-x\right) ^{\alpha -1}\right] \left( g\circ \varphi
\right) (x)dx\right\vert  \label{2.13}
\end{equation}%
\begin{equation*}
\ \ \ \ \ \ \ \ \ \ \ \ \ \ \ \ \ \ \ \ \ \ \ \ \ \ \ \ \ \ \ \ \ \ \ \ \ \
\ \ \ \ \ \ \ \ \ \ \ \ \ \ \ \ \ \ \ \ \ \ \ \ \ \ \ \ \ \ \ \ \ \ \ \ \ \
\ \ \ \ \ \ \ \ \ \ \ \ \ \ \ =\left\vert \int\limits_{1/U(t)}^{1/L(t)}\left[
\left( x-\frac{1}{b}\right) ^{\alpha -1}+\left( \frac{1}{a}-x\right)
^{\alpha -1}\right] \left( g\circ \varphi \right) (x)dx\right\vert
\end{equation*}

for all $t\in [0,1] $.

By (\ref{2.11})- (\ref{2.13}), we have

\begin{equation}
\left\vert \left( \frac{f(a)+f(b)}{2}\right) \left[ J_{1/b^{+}}^{\alpha
}g\circ \varphi (1/a)+J_{1/a^{-}}^{\alpha }g\circ \varphi (1/b)\right] -%
\left[ J_{1/b^{+}}^{\alpha }\left( fg\circ \varphi \right) \left( 1/a\right)
+J_{1/a^{-}}^{\alpha }\left( fg\circ \varphi \right) \left( 1/b\right) %
\right] \right\vert  \label{2.14}
\end{equation}%
\begin{equation*}
\leq \frac{b-a}{4ab\Gamma \left( \alpha \right) }\left\{
\int\limits_{0}^{1}\left\vert \left[ \int\limits_{1/U(t)}^{1/L(t)}\left( x-%
\frac{1}{b}\right) ^{\alpha -1}+\left( \frac{1}{a}-x\right) ^{\alpha -1}%
\right] g\circ \varphi (x)dx\right\vert \left\{ t\left\vert f^{\prime
}(H)\right\vert +(1-t)\left\vert f^{\prime }(a)\right\vert \right\}
(L(t))^{2}dt\ \right.
\end{equation*}%
\begin{equation*}
\ \ \ \ \ \ \ \ \ \ \ \ \ \ \ \ \ \ \ \ \ \ \ \ \ +\left.
\int\limits_{0}^{1}\left\vert \int\limits_{1/U(t)}^{1/L(t)}\left[ \left( x-%
\frac{1}{b}\right) ^{\alpha -1}+\left( \frac{1}{a}-x\right) ^{\alpha -1}%
\right] g\circ \varphi (x)dx\right\vert \left\{ t\left\vert f^{\prime
}(H)\right\vert +(1-t)\left\vert f^{\prime }(b)\right\vert \right\}
(U(t))^{2}dt\right\}
\end{equation*}%
\begin{equation*}
\leq \frac{\left( b-a\right) \left\Vert g\right\Vert _{\infty }}{4ab\Gamma
\left( \alpha \right) }\left\{ \int\limits_{0}^{1}\left[ \int%
\limits_{1/U(t)}^{1/L(t)}\left\vert \left( x-\frac{1}{b}\right) ^{\alpha
-1}+\left( \frac{1}{a}-x\right) ^{\alpha -1}\right\vert dx\right] \left\{
t\left\vert f^{\prime }(H)\right\vert +(1-t)\left\vert f^{\prime
}(a)\right\vert \right\} (L(t))^{2}dt\ \right.
\end{equation*}%
\begin{equation*}
\ \ \ \ \ \ \ \ \ \ \ \ \ \ \ \ \ \ \ \ \ \ \ \ \ +\left.
\int\limits_{0}^{1} \left[ \int\limits_{1/U(t)}^{1/L(t)}\left\vert \left( x-%
\frac{1}{b}\right) ^{\alpha -1}+\left( \frac{1}{a}-x\right) ^{\alpha
-1}\right\vert dx\right] \left\{ t\left\vert f^{\prime }(H)\right\vert
+(1-t)\left\vert f^{\prime }(b)\right\vert \right\} (U(t))^{2}dt\right\} .
\end{equation*}%
In the last inequality,

\begin{equation}
\int\limits_{1/U(t)}^{1/L(t)}\left\vert \left( x-\frac{1}{b}\right) ^{\alpha
-1}+\left( \frac{1}{a}-x\right) ^{\alpha -1}\right\vert
dx=\int\limits_{1/U(t)}^{1/L(t)}\left( x-\frac{1}{b}\right) ^{\alpha
-1}dx+\int\limits_{1/U(t)}^{1/L(t)}\left( \frac{1}{a}-x\right) ^{\alpha -1}dx
\label{2.15}
\end{equation}%
\begin{equation*}
\ \ \ \ \ \ \ \ \ \ \ \ \ \ \ \ \ \ \ \ \ \ \ \ \ \ \ \ \ \ \ \ \ \ \ \ \ \
\ \ \ \ \ \ \ \ \ \ \ \ \ \ \ \ \ =\frac{2^{1-\alpha }}{\alpha }\left( \frac{%
b-a}{ab}\right) ^{\alpha }\left\{ \left( 1+t\right) ^{\alpha }-\left(
1-t\right) ^{\alpha }\right\} .
\end{equation*}

By Lemma 1, we have 
\begin{equation*}
\int\limits_{1/U(t)}^{1/L(t)}\left\vert \left( x-\frac{1}{b}\right) ^{\alpha
-1}+\left( \frac{1}{a}-x\right) ^{\alpha -1}\right\vert
dx=\int\limits_{1/U(t)}^{1/L(t)}\left( x-\frac{1}{b}\right) ^{\alpha
-1}dx+\int\limits_{1/U(t)}^{1/L(t)}\left( \frac{1}{a}-x\right) ^{\alpha -1}dx
\end{equation*}%
\begin{equation*}
\ \ \ \ \ \ \ \ \ \ \ \ \ \ \ \ \ \ \ \ \ \ \ \ \ \ \ \ \ \ \ \ \ \ \ \ \ \
\ \ \ \ \ \ \ \ \ \ \ \ \ \ \ \ \ \leq \frac{2}{\alpha }\left( \frac{b-a}{ab}%
\right) ^{\alpha }t^{\alpha }
\end{equation*}

A combination of (\ref{2.14}) and (\ref{2.15}), we have (\ref{2.9}). This
complete is proof.
\end{proof}

\begin{corollary}
In Corollary 1,

(1)If $\alpha =1$ is in corollary, we obtain following
Hermite-Hadamard-Fejer Type inequality for harmonically-convex function
which is related the left-hand side of (\ref{2.10}):

\begin{equation}
\left\vert \left[ \frac{f(a)+f(b)}{2}\right] \underset{a}{\overset{b}{\int }}%
\frac{g(x)}{x^{2}}dx-\underset{a}{\overset{b}{\int }}f(x)\frac{g(x)}{x^{2}}%
dx\right\vert \leq  \label{2.16}
\end{equation}%
\begin{equation*}
\frac{\left( b-a\right) ^{2}}{4(ab)^{2}}\left\Vert g\right\Vert _{\infty }%
\left[ C_{1}(1)\left\vert f^{\prime }\left( a\right) \right\vert
+C_{2}(1)\left\vert f^{\prime }\left( H\right) \right\vert
+C_{3}(1)\left\vert f^{\prime }\left( b\right) \right\vert \right]
\end{equation*}%
where for $a,b,H>0$, we have 
\begin{equation*}
C_{1}\left( 1\right) =\int\limits_{0}^{1}(1-t)t(L(t))^{2}dt
\end{equation*}%
\begin{equation*}
C_{2}\left( 1\right) =\int\limits_{0}^{1}t^{2}\left[ (L(t))^{2}+(U(t))^{2}%
\right] dt
\end{equation*}%
\begin{equation*}
C_{3}\left( 1\right) =\int\limits_{0}^{1}(1-t)t(U(t))^{2}dt
\end{equation*}

(2)If $g(x)=1 $ is in corollary, we obtain following Hermite-Hadamard-Fejer
Type inequality for harmonically-convex function which is related the
left-hand side of (\ref{2.9}):

\begin{equation}
\left\vert \left( \frac{f(a)+f(b)}{2}\right) -\frac{\left( ab\right)
^{\alpha }\Gamma (\alpha +1)}{2(b-a)^{\alpha }}\left[ J_{1/b^{+}}^{\alpha
}\left( f\circ \varphi \right) \left( 1/a\right) +J_{1/a^{-}}^{\alpha
}\left( f\circ \varphi \right) \left( 1/b\right) \right] \right\vert
\label{2.17}
\end{equation}%
\begin{equation*}
\leq \frac{\left( b-a\right) }{2^{\alpha +2}ab}\left[ C_{1}\left( \alpha
\right) \left\vert f^{\prime }(a)\right\vert +C_{2}\left( \alpha \right)
\left\vert f^{\prime }(H)\right\vert +C_{3}\left( \alpha \right) \left\vert
f^{\prime }(b)\right\vert \right] .
\end{equation*}

(3)If $g(x)=1$ and $\alpha =1$ is in corollary, we obtain following
Hermite-Hadamard-Fejer Type inequality for harmonically-convex function
which is related the left-hand side of (\ref{2.10}):

\begin{equation}
\left\vert \left( \frac{f(a)+f(b)}{2}\right) -\frac{ab}{(b-a)}%
\int\limits_{a}^{b}\frac{f(x)}{x^{2}}dx\right\vert  \label{2.18}
\end{equation}%
\begin{equation*}
\leq \frac{\left( b-a\right) }{4(ab)}\left[ C_{1}\left( 1\right) \left\vert
f^{\prime }(a)\right\vert +C_{2}\left( 1\right) \left\vert f^{\prime
}(H)\right\vert +C_{3}\left( 1\right) \left\vert f^{\prime }(b)\right\vert %
\right] .
\end{equation*}
\end{corollary}

\begin{theorem}
Let $f:I\subseteq 
%TCIMACRO{\U{211d} }%
%BeginExpansion
\mathbb{R}
%EndExpansion
\backslash \{0\}\longrightarrow 
%TCIMACRO{\U{211d} }%
%BeginExpansion
\mathbb{R}
%EndExpansion
$ be differentiable mapping $I^{o}$, where $a,b\in I$ with $a<b$. If the
mapping $\left\vert f^{\prime }\right\vert ^{q}$ is harmonically-convex on $%
\left[ a,b\right] $, then the following inequality holds:

\begin{equation}
\left\vert \left[ h(b)-2h(a)\right] \frac{f(a)}{2}+h(b)\frac{f(b)}{2}%
-\int\limits_{a}^{b}f(x)h^{\prime }(x)dx\right\vert  \label{2.19}
\end{equation}%
\begin{equation*}
\leq \frac{b-a}{4ab}\left\{ 
\begin{array}{c}
\left( \underset{0}{\int\limits^{1}}\left\vert 2h(L(t))-h(b)\right\vert
dt\right) ^{1-\frac{1}{q}}\times \\ 
\left( \underset{0}{\overset{1}{\int }}%
\begin{array}{c}
\left( \left\vert 2h(L(t))-h(b)\right\vert dt\right) \\ 
\times \left( t\left( L(t)\right) ^{2q}\left\vert f^{\prime }(a)\right\vert
^{q}+(1-t)\left( L(t)\right) ^{2q}\left\vert f^{\prime }(H)\right\vert
^{q}\right)%
\end{array}%
\right) ^{\frac{1}{q}}%
\end{array}%
\right.
\end{equation*}%
\begin{equation*}
\ \ \ \ \ \ \ \ \ \ \ \ \ \ \ \ \ \ \ \ \ \ \ \ \ \ \ \ \ \ \ \ \ \ \ \ \ \
\left. 
\begin{array}{c}
+\ \left( \underset{0}{\int\limits^{1}}\left\vert 2h(U(t))-h(b)\right\vert
dt\right) ^{1-\frac{1}{q}}\times \\ 
\left( \underset{0}{\overset{1}{\int }}%
\begin{array}{c}
\left( \left\vert 2h(U(t))-h(b)\right\vert dt\right) \\ 
\times \left( t\left( U(t)\right) ^{2q}\left\vert f^{\prime }(b)\right\vert
^{q}+(1-t)\left( U(t)\right) ^{2q}\left\vert f^{\prime }(H)\right\vert
^{q}\right)%
\end{array}%
\right) ^{\frac{1}{q}}%
\end{array}%
\right\}
\end{equation*}
\end{theorem}

\begin{proof}
Continuing from (\ref{2.7}) in Theorem 7, we use H\"{o}lder Inequality and
we use that $\left\vert f^{\prime }\right\vert ^{q}$ is harmonically-convex.
Thus this proof is complete.
\end{proof}

\begin{corollary}
Let $h(t)=\int\limits_{1/t}^{1/a}\left[ \left( x-\frac{1}{b}\right) ^{\alpha
-1}+\left( \frac{1}{a}-x\right) ^{\alpha -1}\right] \left( g\circ \varphi
\right) (x)dx$ for all $t\in \lbrack a,b]$ and $g:\left[ a,b\right]
\longrightarrow \left[ 0,\infty \right) $ be continuous positive mapping and
symmetric to $\frac{2ab}{a+b}$in Teorem 8, we obtain:

\begin{equation}
\left\vert \left( \frac{f(a)+f(b)}{2}\right) \left[ J_{1/b^{+}}^{\alpha
}\left( g\circ \varphi \right) (1/a)+J_{1/a^{-}}^{\alpha }\left( g\circ
\varphi \right) (1/b)\right] -\left[ J_{1/b^{+}}^{\alpha }\left( fg\circ
\varphi \right) \left( 1/a\right) +J_{1/a^{-}}^{\alpha }\left( fg\circ
\varphi \right) \left( 1/b\right) \right] \right\vert  \label{2.20}
\end{equation}

\begin{equation*}
\leq \frac{(b-a)^{\alpha +1}\left\Vert g\right\Vert _{\infty }}{2^{\alpha
+1}(ab)^{\alpha +1}\Gamma \left( \alpha +1\right) }\left( \frac{%
2^{2}(2^{\alpha }-1)}{\alpha +1}\right) ^{1-\frac{1}{q}}\left[ C_{1}\left(
\alpha ,q\right) \left\vert f^{\prime }(a)\right\vert ^{q}+C_{2}\left(
\alpha ,q\right) \left\vert f^{\prime }(H)\right\vert ^{q}+C_{3}\left(
\alpha ,q\right) \left\vert f^{\prime }(b)\right\vert ^{q}\right] ^{\frac{1}{%
q}}
\end{equation*}%
where for $q>1$

\begin{equation*}
C_{1}\left( \alpha ,q\right) =\overset{1}{\underset{0}{\int }}\left[
(1+t)^{\alpha }-(1-t)^{\alpha }\right] t\left( L(t)\right) ^{2q}dt
\end{equation*}%
\begin{equation*}
C_{2}\left( \alpha ,q\right) =\overset{1}{\underset{0}{\int }}\left[
(1+t)^{\alpha }-(1-t)^{\alpha }\right] (1-t)\left( \left( L(t)\right)
^{2q}+\left( U(t)\right) ^{2q}\right) dt
\end{equation*}%
\begin{equation*}
C_{3}\left( \alpha ,q\right) =\overset{1}{\underset{0}{\int }}\left[
(1+t)^{\alpha }-(1-t)^{\alpha }\right] t\left( U(t)\right) ^{2q}dt.
\end{equation*}

\begin{proof}
Continuing from (\ref{2.15}) of Corollary 1 and (\ref{2.19}) in Theorem 8,

\begin{equation}  \label{2.21}
\left\vert \left( \frac{f(a)+f(b)}{2}\right) \left[ J_{a^{+}}^{\alpha
}g(b)+J_{b^{-}}^{\alpha }g(a)\right] -\left[ J_{a^{+}}^{\alpha }\left(
fg\right) \left( b\right) +J_{b^{-}}^{\alpha }\left( fg\right) \left(
a\right) \right] \right\vert
\end{equation}

\begin{equation*}
\leq \frac{\left( b-a\right) ^{\alpha +1}}{2^{\alpha +1}\Gamma \left( \alpha
+1\right) }\left\{ 
\begin{array}{c}
\left( \underset{0}{\overset{1}{\int }}\left[ (1+t)^{\alpha }-(1-t)^{\alpha }%
\right] dt\right) ^{1-\frac{1}{q}}\times \\ 
\left( \underset{0}{\overset{1}{\int }}\left[ (1+t)^{\alpha }-(1-t)^{\alpha }%
\right] \left( t\left( L(t)\right) ^{2q}\left\vert f^{\prime }(a)\right\vert
^{q}+(1-t)\left( L(t)\right) ^{2q}\left\vert f^{\prime }(H)\right\vert
^{q}\right) dt\right) ^{\frac{1}{q}}%
\end{array}%
\right.
\end{equation*}%
\begin{equation*}
\ \ \ \ \ \ \ \ \ \ \ \ \ \ \ \ \ \ \ \ \ \ \ \ \ \ \ \ \ \ \ \ +\left. 
\begin{array}{c}
\left( \underset{0}{\overset{1}{\int }}\left[ (1+t)^{\alpha }-(1-t)^{\alpha }%
\right] dt\right) ^{1-\frac{1}{q}}\times \\ 
\left( \underset{0}{\overset{1}{\int }}\left[ (1+t)^{\alpha }-(1-t)^{\alpha }%
\right] \left( t\left( U(t)\right) ^{2q}\left\vert f^{\prime }(b)\right\vert
^{q}+(1-t)\left( U(t)\right) ^{2q}\left\vert f^{\prime }(H)\right\vert
^{q}\right) dt\right) ^{\frac{1}{q}}%
\end{array}%
\right\}
\end{equation*}%
\begin{equation*}
\leq \frac{(b-a)^{\alpha +1}\left\Vert g\right\Vert _{\infty }}{2^{\alpha
+1}(ab)^{\alpha +1}\Gamma \left( \alpha +1\right) }\left( \frac{2^{\alpha
+1}-2}{\alpha +1}\right) ^{1-\frac{1}{q}}\left[ 
\begin{array}{c}
\left( \underset{0}{\overset{1}{\int }}%
\begin{array}{c}
\left[ (1+t)^{\alpha }-(1-t)^{\alpha }\right] \times \\ 
\left[ t\left( L(t)\right) ^{2q}\left\vert f^{\prime }(a)\right\vert
^{q}+(1-t)\left( L(t)\right) ^{2q}\left\vert f^{\prime }(H)\right\vert ^{q}%
\right]%
\end{array}%
dt\right) ^{\frac{1}{q}} \\ 
+\left( \underset{0}{\overset{1}{\int }}%
\begin{array}{c}
\left[ (1+t)^{\alpha }-(1-t)^{\alpha }\right] \times \\ 
\left[ t\left( U(t)\right) ^{2q}\left\vert f^{\prime }(b)\right\vert
^{q}+(1-t)\left( U(t)\right) ^{2q}\left\vert f^{\prime }(H)\right\vert ^{q}%
\right]%
\end{array}%
dt\right) ^{\frac{1}{q}}%
\end{array}%
\right]
\end{equation*}

By the power-mean inequality $\left( a^{r}+b^{r}<2^{1-r}\left( a+b\right)
^{r}for\quad a>0,b>0,\quad r<1\right) $ and $\frac{1}{p}+\frac{1}{q}=1$ we
have

\begin{equation}
\leq \frac{(b-a)^{\alpha +1}\left\Vert g\right\Vert _{\infty }}{2^{\alpha
+1}(ab)^{\alpha +1}\Gamma \left( \alpha +1\right) }\left( \frac{%
2^{2}(2^{\alpha }-1)}{\alpha +1}\right) ^{\frac{1}{p}}\left[ \overset{1}{%
\underset{0}{\int }}\left( 
\begin{array}{c}
\left[ (1+t)^{\alpha }-(1-t)^{\alpha }\right] t\left( L(t)\right)
^{2q}\left\vert f^{\prime }(a)\right\vert ^{q}+ \\ 
\left[ (1+t)^{\alpha }-(1-t)^{\alpha }\right] (1-t)\left( 
\begin{array}{c}
\left( L(t)\right) ^{2q} \\ 
+\left( U(t)\right) ^{2q}%
\end{array}%
\right) \left\vert f^{\prime }(H)\right\vert ^{q} \\ 
+\left[ (1+t)^{\alpha }-(1-t)^{\alpha }\right] t\left( U(t)\right)
^{2q}\left\vert f^{\prime }(b)\right\vert ^{q}%
\end{array}%
\right) dt\right] ^{\frac{1}{q}}  \label{2.22}
\end{equation}
\end{proof}

\begin{corollary}
When $\alpha =1$ and $g(x)=1$ is taken in Corollary 3, we obtain:

\begin{equation}
\left\vert \left( \frac{f(a)+f(b)}{2}\right) -\frac{ab}{(b-a)}%
\int\limits_{a}^{b}\frac{f(x)}{x^{2}}dx\right\vert  \label{2.23}
\end{equation}

\begin{equation*}
\leq \frac{\left( b-a\right) }{2^{2+\frac{1}{q}}\left( ab\right) }\left[
C_{1}\left( 1,q\right) \left\vert f^{\prime }(a)\right\vert ^{q}+C_{2}\left(
1,q\right) \left\vert f^{\prime }(H)\right\vert ^{q}+C_{3}\left( 1,q\right)
\left\vert f^{\prime }(b)\right\vert ^{q}\right] ^{\frac{1}{q}}
\end{equation*}
\end{corollary}

This proof is complete.
\end{corollary}


\begin{thebibliography}{99}
\bibitem{CW141} F. Chen and S. Wu, Hermite-Hadamard type inequalities for
harmonically convex functions,Journal of Applied Mathematics Volume 2014
(2014), Article ID 386806.

\bibitem{ZD10} Z. Dahmani, On Minkowski and Hermite-Hadamard integral
inequalities via fractional integration, Ann. Funct. Anal. 1(1) (2010), pp.
51-58.

\bibitem{SSD12} S. S. Dragomir, Hermite Hadamard' s type inequalities for
convex functions of selfadjoint operators in Hilbert spaces, Linear Algebra
Appl. 436 (2012), no.5, 1503-1515.

\bibitem{H11} D-Y. Hwang, Some inequalities for differentiable convex
mapping with application to weighted trapezoidal formula and higher moments
of random variables, Applied Mathematics and Computation, 217 (2011),
9598-9605.

\bibitem{IK215} I. Iscan, M. Kunt, Fej\'{e}r and Hermite-Hadamard-Fej\'{e}r
type inequalities for harmonically $s$-convex functions via Fractional
Integrals, The Australian Journal of Mathematical Analysis and Applications,
(2015), Vol: 12, 1 ,Article 10, pp 1-6.

\bibitem{I113} I. Iscan, Hermite-Hadamard type inequalities for harmonically
convex functions, Hacettepe Journal of Mathematics and Statistics 43 (6)
(2014), 935-942.

\bibitem{I115} I. Iscan, Ostrowski type inequalities for harmonically $s$%
-convex functions, Konuralp Jurnal of Mathematics, Volume 3, No 1 (2015),
pp. 63-74..

\bibitem{I114} I. Iscan, Some new general integral inequalities for h-convex
and h-concave functions, Adv. Pure Appl. Math. 5(1) (2014), pp. 21-29 . doi:
10.1515/apam-2013-0029.

\bibitem{I213} I. Iscan, Generalization of different type integral
inequalitiesfor s-convex functions via fractional integrals, Applicable
Analysis, 2013. doi: 10.1080/00036811.2013.851785.

\bibitem{IK15} I. Iscan, M. Kunt, Hermite-Hadamard-Fejer type inequalities
for harmonically convex functions via fractional integrals, RGMIA Research
Report Collection, 18(2015), Article 107, pp. 1-16.

\bibitem{IW14} I. Iscan, S. Wu Hermite-Hadamard type inequalities for
harmonically-convex functions via fractional integrals, Applied Mathematics
and Computation, 238 (2014), 237--244.

\bibitem{IS15} \.{I}. \.{I}\c{s}can, S. Turhan, Generalized
Hermite-Hadamard-Fejer type inequalities for GA-convex functions via
Fractional integral, arXiv:1511.03308v1 [math.CA], 10 Nov 2015.

\bibitem{ML14} M.A. Latif, New Hermite Hadamard type integral inequalities
for GA-convex functions with applications. Volume 34, Issue 4 (Nov 2014).

\bibitem{LF06} L. Fej\'{e}r, Uberdie Fourierreihen, II, Math. Naturwise.
Anz. Ungar. Akad. , Wiss, 24 (1906), pp. 369-390, (in Hungarian)

\bibitem{MS12} M. Z. Sar\i kaya, On new Hermite Hadamard Fej\'{e}r type
integral inequalities, Stud. Univ. Babe\c{s}-Bolyai Math., 57(3) (2012), pp.
377--386.

\bibitem{SSYB13} M. Z. Sar\i kaya, E. Set, H. Yald\i z and N. Ba\c{s}ak,
Hermite-Hadamard's inequalities for fractional integrals and related
fractional inequalities, Mathematical and Computer Modelling, 57(9) (2013),
pp. 2403-2407.
\end{thebibliography}
\end{document}